\documentclass[12pt,reqno,oneside]{amsart}
\usepackage[a4paper,marginparwidth=2.5cm]{geometry}
\usepackage[utf8]{inputenc}
\usepackage{amsthm}
\usepackage{amsmath}
\usepackage{amsfonts}
\usepackage{mathtools}
\usepackage{enumitem}
\usepackage{marginnote}
\usepackage[dvipsnames]{xcolor}
\usepackage{hyperref}
\usepackage{amssymb}

\setlist[enumerate]{label=(\alph*)}
\numberwithin{equation}{section}

\makeatletter
\renewcommand\subsection{\@startsection{subsection}{2}%
  \z@{-.5\linespacing\@plus-.7\linespacing}{.5\linespacing}%
  {\normalfont\scshape}}
\renewcommand\subsubsection{\@startsection{subsubsection}{3}%
  \z@{.5\linespacing\@plus.7\linespacing}{-.5em}%
  {\normalfont\scshape}}
\makeatother

\definecolor{Darkgreen}{rgb}{0,0.4,0}
\hypersetup{
  colorlinks,
  linkcolor=Darkgreen,
  citecolor=Darkgreen,
  breaklinks=true,
  bookmarksnumbered=true,
  plainpages=false,
  unicode=true,
  pdfauthor={Jiří Černý, Flavio Dalessi},
  pdftitle={Tightness of the maximum of branching random walk in random
    environment and zero-crossings of solutions to discrete parabolic differential equations}
}

\newtheorem{theorem}{Theorem}[section]
\newtheorem{proposition}[theorem]{Proposition}
\newtheorem{lemma}[theorem]{Lemma}
\newtheorem{corollary}[theorem]{Corollary}
\newtheorem{claim}[theorem]{Claim}

\theoremstyle{definition}

\theoremstyle{remark}
\newtheorem{remark}[theorem]{Remark}

\newcommand{\R}{\mathbb{R}}

\newcommand{\N}{\mathbb{N}}
\newcommand{\Z}{\mathbb{Z}}

\newcommand{\PP}{\mathbb{P}}
\newcommand{\E}{\mathbb{E}}
\newcommand{\OZ}{\overline{\mathbb{Z}}}

\newcommand{\es}{\mathbf{es}}
\newcommand{\ei}{\mathbf{ei}}
\newcommand{\bbone}{\boldsymbol 1}

\let\phi\varphi

\DeclareMathOperator{\sgn}{sgn}

\DeclareMathOperator{\essinf}{ess\,inf}
\DeclareMathOperator{\esssup}{ess\,sup}

\DeclarePairedDelimiter{\ceil}{\lceil}{\rceil}
\DeclarePairedDelimiter{\floor}{\lfloor}{\rfloor}
\DeclarePairedDelimiter{\norm}{\lVert}{\rVert}

\DeclarePairedDelimiterX{\inp}[2]{\langle}{\rangle}{#1, #2}

\begin{document}

\title[Tightness of the maximum of BRWRE and zero-crossings]
{Tightness of the maximum of branching random walk in random environment
  and zero-crossings of solutions to discrete parabolic differential equations}

\author{Jiří Černý, Flavio Dalessi}

\begin{abstract}
    We study branching random walk on $\mathbb Z$ in a bounded i.i.d.~random
    environment. For this process, we prove that, for almost every realization
    of the environment, the distributions of the maximally displaced particle
    (re-centered around their medians) are tight. This extends the result of
    \cite{Kri24}, where tightness was established in the annealed sense, and
    of \cite{CDO25}, where a similar quenched result was proved for branching
    Brownian motion in random environment. Our proof relies on studying
    certain discrete-space linear PDEs and establishing that the number of
    zero-crossings of their solutions is non-increasing in time. We observe
    that our technique requires no additional assumptions on the environment,
    in contrast to \cite{Kri24,CDO25}.
\end{abstract}

\maketitle

\section{Introduction}
\label{sec:introduction}

The purpose of this paper is twofold. The first objective is to establish
quenched tightness of the maximum of a branching random walk in a random
environment (BRWRE), that is, to show that, for almost every realization of
the environment, the distributions of the maximally displaced particle at
time $t\geq 0$---when re-centered around their respective quenched
medians---form a tight family.

The second objective is to analyze the solutions to the
discrete parabolic equation
\begin{equation}
  \label{eq:main}
  \partial_tu(t,x)
  = \frac{1}{2} \Delta_d u(t,x) - \kappa(t,x)u(t,x), \quad x \in \Z, \ t >0,
\end{equation}
for a given bounded measurable $\kappa : [0,\infty) \times \Z \rightarrow \R$,
and show that the number of their zero-crossings decreases in time.
This analysis is primarily motivated by its role in the proof of the
tightness result. By controlling the zero-crossings of the solutions to
\eqref{eq:main}, we can adapt the arguments of \cite{CDO25}---where quenched
tightness of the maximum of branching Brownian motion in a random
environment (BBMRE) was established---and deduce tightness for the maximum of
the BRWRE.

We note that the monotonicity of the number of zero-crossings is a property
well known to hold for a broad class of one-dimensional second-order linear
parabolic partial differential equations on the \emph{real line} (see, for
  example, \cite{Ang88}, \cite{Nad15}). To prove that this monotonicity
property also holds in the discrete setting, we adapt the probabilistic
arguments developed for the real-line setting in \cite{EW99}.

We briefly describe the organization of the paper. To make our main objective
precise, in Section~\ref{sec:tightness_result} we introduce the BRWRE and
state our first result: tightness of its maximum. In Section~\ref{sec:model},
we state our second main result, which characterizes the zero-crossings of
the solution to \eqref{eq:main}. By using this analytical tool, in
Sections~\ref{sec:F_KPP}, \ref{sec:tilting} and
\ref{sec:proof_prop_and_lemma}, we prove tightness of the maximum of the
BRWRE. The proof of the zero-crossings result is given in
Section~\ref{sec:proof_zero_crossings}. More details on the structure of
these proofs will be given in Section~\ref{sec:F_KPP}.

\section{Tightness of the maximum of branching random walk in random environment}
\label{sec:tightness_result}

This section is dedicated to a precise formulation of our main result, which
establishes tightness, after appropriate re-centering, of the maximum of the
BRWRE. We adopt the same setting as in \cite{CD20}, where an invariance
principle for the maximum of the BRWRE was proved. Specifically, we fix a
random environment consisting of an i.i.d.~family $(\xi(x))_{x \in \Z}$ of
random variables satisfying
\begin{equation}
  \label{eq:ellipticity}
  0< \ei \coloneq \essinf \xi(0)
  < \esssup \xi(0) \eqcolon \es< \infty.
\end{equation}
Given a fixed starting point $x \in \Z$ and a realization of the environment,
we consider the following continuous-time branching process: At time zero we
initialize the process by placing a particle at $x$, which then moves
according to a continuous-time simple random walk with jump rate one and,
independently, branches into two distinct particles at rate $\xi(y)$ when
located at a site $y \in \Z$. Upon branching, the offspring particles evolve
as independent and identically distributed copies, following the same
diffusion and branching mechanism as the original particle. We write
$\mathtt{P}_x^\xi$ for the law of this branching process and
$\mathtt{E}_x^\xi$ for the corresponding expectation. The law of the
environment and its expectation are denoted by $\PP$ and $\E$.

The BRWRE obtained with this construction represents the discrete-space
analog of the BBMRE considered in \cite{CDO25}. While the results in
\cite{CDO25} are presented for a general offspring distribution, we restrict
our attention to the binary branching case for simplicity, similarly to
\cite{CD20}. Nevertheless, our analysis can readily be extended to the
general setting.

Let $M(t)$ denote the position at time $t \geq 0$ of the rightmost particle
of the BRWRE, and
$m(t)\coloneq \sup\{y \in \Z: \mathtt{P}_x^\xi(M(t)\geq y) \geq
  \frac{1}{2}\}$
be its median under the measure $\mathtt{P}_x^\xi$. Note
that the median $m(t)$ is random, as it depends on the realization of the
environment $\xi$. Our tightness result reads as follows.

\begin{theorem}\label{thm:tightness}
    For $\mathbb P$-almost every realization of the environment, the family
    $(M(t)- m(t) )_{t \geq 0}$ is tight under $\mathtt{P}_0^\xi$.
\end{theorem}

Tightness of the maximum of the BRWRE has recently been established in the
annealed sense in \cite{Kri24}. We address the (stronger) notion of quenched
tightness, which previously had been proved only along subsequences, see
\cite{Kri21}.

An important feature of Theorem \ref{thm:tightness} is that no assumption is
made on the distribution of the random environment, besides
\eqref{eq:ellipticity}. In particular, in contrast to \cite{Kri24, CDO25}
(but not to \cite{Kri21}), we do not require that the corresponding Lyapunov
exponent is strictly concave at the asymptotic speed of the maximum, see
Remark~\ref{rk:v_0_vs_v_c} below for a precise statement of this assumption.
For now, note that under this assumption it is possible to define a certain
tilted measure under which a single particle moves with the speed of the
maximum. This measure, first introduced in \cite{CD20}, is at the core of
many results related to the BRWRE, and is likewise central to our arguments.

In the course of proving annealed tightness, Kriechbaum \cite{Kri24} derives
formulas for the centering of the maximum and bounds on the decay of the
probabilities that the maximum deviates significantly from its median. Our
approach does not directly yield estimates of this type.

Finally, we note that it can easily be verified that our observations on the
role of the strict concavity of the Lyapunov exponent in the proof of
Theorem~\ref{thm:tightness} (see Sections~\ref{sec:F_KPP}, \ref{sec:tilting}
  and \ref{sec:proof_prop_and_lemma}) also extend to the continuous-space
setting of \cite{CDO25}. Consequently, Theorem 2.1 in
  \cite{CDO25}, establishing tightness of the maximum of the BBMRE, remains
  valid even when Assumption 3 in \cite{CDO25}---the strict concavity at the
  asymptotic speed---is removed.

\section{Monotonicity of the number of zero-crossings}
\label{sec:model}

This section presents our second main result, describing the zero-crossing
behavior of solutions to \eqref{eq:main}, which plays a central role in the
proof of Theorem~\ref{thm:tightness}. In order to state this result, we
define the number of zero-crossings of a function~$f : \Z \rightarrow \R$  by
\begin{equation*}
  \Sigma(f) \coloneq 0 \vee \sup\{ n \geq 1: \exists \ x_1<  \dots < x_{n+1}
    \text{ s.t. } f(x_i)f(x_{i+1})<0 \text{ for } 1 \leq i \leq n \},
\end{equation*}
and recall that the discrete Laplace operator $\Delta_d$ is
defined, for $f: \Z \rightarrow \R$, by
\begin{equation*}
    \Delta_d f(x)=f(x+1)-2f(x)+f(x-1), \quad  x \in \Z.
\end{equation*}

\begin{theorem}
  \label{Thm:zero-crossings}
    For $\kappa \in L^\infty([0,\infty) \times \Z)$ and $u_0 \in \ell^1(\Z)$,
    let $u: [0,\infty) \times \Z \rightarrow \R$ be a solution of
    \begin{equation}
      \label{eq:main_2}
      \begin{aligned}
        \frac{\partial}{\partial t} u(t,y) &
        = \frac{1}{2} \Delta_d u(t,y) - \kappa(t,y)u(t,y), \quad
        &&t > 0, \ y \in \Z,
        \\ u(0,y)&=u_0(y),
        &&y \in \Z.
      \end{aligned}
    \end{equation}
    Then, for all  $0 \leq s \leq t$,
    \begin{equation}
      \label{eq:u_t_vs_u_s}
      \Sigma(u(t,\cdot)) \leq \Sigma(u(s,\cdot)).
    \end{equation}
    Moreover, in the special case  $\Sigma(u_0)= 1$,
    if $u_0$ satisfies
    \begin{equation} \label{eq:cond_u_0}
      \{y : u_0(y)<0\} = \{y \in \Z: y \leq a_0\}
      \text{ and }
      \{y : u_0(y)>0\} = \{y \in \Z: y \geq b_0\}
    \end{equation}
    for some $a_0<b_0 \in \Z$, then, for all $t \geq 0$, the same property is
    satisfied by $u(t,\cdot)$ for some
    $a_t<b_t \in \Z \cup \{-\infty, \infty\}$.
\end{theorem}

Note that the second part of Theorem~\ref{Thm:zero-crossings} is not an
immediate consequence of the first. Indeed, $\Sigma(u(t, \cdot)) \leq 1$ does
not preclude the existence of some $y \in \Z$ satisfying $u(t,y)=0$ and
$u(t,y \pm 1)>0$.

Although in our applications of Theorem~\ref{Thm:zero-crossings} the initial
condition $u_0$ always satisfies $\Sigma(u_0)=1$ (see Section
 ~\ref{sec:F_KPP}), we state the result in the more general setting of
potentially multiple zero-crossings, as it might be of independent interest.

Continuous-space analogs of Theorem~\ref{Thm:zero-crossings} have a long
history in the analysis literature. Results in this direction were first
established in the 19th century by Sturm \cite{Stu36}. His ideas were later
revived in the study of linear and nonlinear parabolic equations (see, for
  example, \cite{Ang88}, \cite{Ang91}, \cite{DGM14} and \cite{Nad15}); see
also \cite{Gal04} for a detailed discussion of the Sturmian principle and its
applications. In the context of differential equations arising from a
branching Brownian motion, a related result appeared already in the 1930s in the
study of the F-KPP equation by Kolmogorov, Petrovskii and Piskunov
\cite{KPP37}.

A version of Theorem~\ref{Thm:zero-crossings}---describing the evolution of
zero-crossings when operators related to certain time-homogeneous Markov
processes on the integers are considered---was established already in the
1950s (see \cite{KM57}). The methods employed by the authors highlight,
through the Karlin--McGregor determinant formula of coincidence probabilities
for multiple particle systems, a probabilistic interpretation of the minors
of the transition matrix of the process in question (see \cite{KM58}). This
formula was later extended to a larger family of time-inhomogeneous Markov
processes (see \cite{Kar88}). We will discuss in Section
\ref{sec:proof_zero_crossings} the connection between these results and the
proof of Theorem~\ref{Thm:zero-crossings}.

\section{Branching random walk and the randomized F-KPP equation}
\label{sec:F_KPP}

We now return to the setting of Section~\ref{sec:tightness_result} and prove
tightness of the maximum of the BRWRE. As anticipated, we prove Theorem
\ref{thm:tightness} by adapting to the discrete-space framework the ideas in
\cite{CDO25}.

The plan for the proof of Theorem~\ref{thm:tightness} is as follows. We begin
by recalling some results on the BRWRE and its connection to the F-KPP
equation. The probabilistic representation of solutions to the F-KPP equation
via the BRWRE underlies our proof of Theorem~\ref{thm:tightness}, and is
likewise central to the arguments in \cite{CDO25}. Next, we present two
auxiliary results---Lemmas~\ref{lemma:wave_time}
and~\ref{lemma:wave}---which, together with Theorem~\ref{Thm:zero-crossings},
constitute the key components of the proof of tightness. We conclude
Section~\ref{sec:F_KPP} by explaining how Theorem~\ref{thm:tightness} can be
deduced from these statements. In Section~\ref{sec:tilting}, we analyze a
particular change of measure that is essential for studying large deviations
of the maximum of the BRWRE. The arguments in this section deviate
considerably from their continuous-space counterparts in \cite{CDO25}.
Indeed, since the change of measure behaves quite differently in the discrete
and continuous settings, distinct techniques are required. Building on the
results of Section~\ref{sec:tilting}, Section~\ref{sec:proof_prop_and_lemma}
establishes Lemmas~\ref{lemma:wave_time} and~\ref{lemma:wave} by adapting the
arguments used in \cite{CDO25} for their continuous-space analogs
(Corollary~3.6 and Lemma~6.1 therein). Finally, the proof of the
zero-crossings result, Theorem~\ref{Thm:zero-crossings}, is presented in
Section~\ref{sec:proof_zero_crossings}. The steps ensuring that
Theorem~\ref{thm:tightness} holds even without strict concavity of the
Lyapunov exponent at the asymptotic speed are contained in
Claims~\ref{claim:wave}, \ref{claim:v_2} and Lemma~\ref{lemma:velcrit}.

\subsection{The randomized F-KPP equation}

In this section, we study the solution of the randomized F-KPP equation and
its relation to the BRWRE.

We begin by recalling that, given an initial condition
$w_0: \Z \rightarrow [0,1]$, the randomized F-KPP equation
\begin{equation}
  \label{eq:F-KPP}
  \begin{aligned}
    \partial_t w(t,x)
    &= \frac{1}{2} \Delta_d w(t,x) + \xi(x)w(t,x)(1-w(t,x)), \quad
    &&t > 0, \ x \in \Z,
    \\ w(0,x)&=w_0(x),  &&x \in \Z,
  \end{aligned}
\end{equation}
admits a unique non-negative solution.
Note that the F-KPP equation can be viewed as an instance of
\eqref{eq:main_2}, with $\kappa(t,x)=\xi(x)(w(t,x)-1)$. This observation will
later be used to apply Theorem~\ref{Thm:zero-crossings} to the difference
$w-w'$ between two solutions $w$, $w'$ to \eqref{eq:F-KPP} with different
initial conditions.

We now recall a few facts from \cite{CD20}. As there, we use $N(t,y)$ to
denote the number of particles in the BRWRE that are located at $y \in \Z$ at
time $t \geq 0$, and write
\begin{equation*}
  N^\geq(t,y)\coloneq \sum_{z \geq y} N(t,z)
\end{equation*}
for the number of particles located to the right of $y$ at time $t$.
The next proposition recalls the well-known connection
of the solution to \eqref{eq:F-KPP} to the
BRWRE. (In its statement we use the convention $0^0 = 1$.)

\begin{proposition}[Proposition 7.1 in \cite{CD20}]
  For each $w_0: \Z \rightarrow[0,1]$,
  \begin{equation}\label{eq:McKean}
    w(t,x)\coloneq 1- \mathtt{E}_x^\xi\bigg[
      \prod_{z\in \Z}(1-w_0(z))^{N(t,z)}\bigg]
  \end{equation}
  solves \eqref{eq:F-KPP}. In particular, for every $y\in \mathbb Z$,
  \begin{equation}\label{eq:w^y}
    w^y(t,x) \coloneq \mathtt{P}_x^\xi(M(t)\geq y)
  \end{equation}
  solves \eqref{eq:F-KPP} with the initial condition
  $w_0^y\coloneq\bbone_{\{y,y+1,\dots \}}$.
\end{proposition}

In what follows, we will make extensive use of \eqref{eq:w^y} to study both
the F-KPP equation and the maximum of the BRWRE.

From now on, let $X=(X_t)_{t \geq 0}$ denote a one-dimensional
continuous-time simple random walk with rate one, which is started at
$x \in \Z$ under the measure $P_x$.
The corresponding expectation is denoted with $E_x$.
By the Feynman--Kac formula, the expected
number of particles $\mathtt{E}_x^\xi[N(t,y)]$ can be expressed in terms of
the random walk $X$, namely
\begin{equation}
  \label{eq:Feynman_Kac}
  \mathtt{E}^{\xi}_x[N(t,y)]
  =E_{x}\Big[ \exp\Big\{\int_0^t \xi(X_s) ds\Big\}; X_t= y\Big]
\end{equation}
for each $x,y \in \Z$, $t \geq 0 $ and $\xi$.  Generalizations of
\eqref{eq:Feynman_Kac} can be found in Proposition~3.1 in \cite{CD20}.

Another key object related to our problem is the (quenched) Lyapunov exponent
$\lambda: \R \rightarrow \R$, given by
\begin{equation*}
  \lambda(v) \coloneq \lim_{t \rightarrow \infty}
  \frac{1}{t}\log \mathtt{E}_0^\xi[N(t,\floor{ vt })].
\end{equation*}
By Proposition A.3 in \cite{CD20}, $\lambda$ is well defined, non-random,
even and concave. Moreover, there exists a unique $v_0 \in (0,\infty)$ such that
\begin{equation*}
  \lambda(v_0)=0, \quad \PP\text{-a.s.}
\end{equation*}
Furthermore, there exists a unique
\begin{equation}\label{eq:def_v_c}
  v_c \in (0,\infty)
\end{equation}
such that $\lambda$ is linear on $[0,v_c]$ and strictly concave on the
interval $(v_c, \infty)$. The role of $v_0$ and $v_c$ in the proof of
Theorem~\ref{thm:tightness} will be clarified in the next section.

\begin{remark}
  \label{rk:v_0_vs_v_c}
  The arguments of \cite{CD20, Kri24} (and several other recent papers on
    BRWRE, BBMRE and the F-KPP equation) rely heavily on the assumption
  \begin{equation}
    \label{eqn:vo_vc}
    v_0 > v_c.
  \end{equation}
  This assumption enables the introduction of certain tilted measures (see
    Section~\ref{sec:tilting} below) which are indispensable for studying the
  behavior of the maximum of the BRWRE. As explained in the introduction,
  although we also use these tilted measures in this paper, our main result,
  Theorem~\ref{thm:tightness}, holds without this assumption and relies
  solely on the ellipticity condition \eqref{eq:ellipticity}.
\end{remark}

We conclude this section by presenting two important results which, together
with Theorem~\ref{Thm:zero-crossings}, lead to Theorem~\ref{thm:tightness}.
The first is a discrete-space analog of Corollary~3.6 in \cite{CDO25}.
Roughly speaking, it states that if the BRWRE has probability at least
$\varepsilon$ to reach $y$ by time $t$, then it also has probability at least
$1-\varepsilon$ to reach $y$ by time $t+u_\varepsilon$, with $u_\varepsilon$
being independent of $y$. The result is stated in terms of solutions to the
F-KPP equation by means of \eqref{eq:w^y}.

\begin{lemma}
  \label{lemma:wave_time}
  For every $\varepsilon \in (0,1/2)$, there exists
  $u = u(\varepsilon) \in (0,\infty)$ such that, $\PP$-a.s, for
  all $t\geq0$ and $y \in \Z$ it holds that
  \begin{equation}
    \label{eq:lemma_wave_time}
    w^y(t,0)\geq \varepsilon \text{ implies }
    w^y(t+t',0) \geq 1- \varepsilon \text{ for all } t' \geq u.
  \end{equation}
\end{lemma}

The second result constitutes the backbone of the proof of
Theorem~\ref{thm:tightness} and can be viewed as a discrete-space analog of
Lemma~6.1 in \cite{CDO25}. Informally, the result asserts that if the BRWRE
is started at $z$, then the probability that the process reaches
$z +vt$ by time $t$ is larger than the probability that it reaches
$z+vt+ \Delta_{u,v}$ by time $t+u$, provided that $\Delta_{u,v}>0$ is large
enough. The interesting feature of the result is that, for given $u,v>0$, the
same $\Delta_{u,v}$ applies to any (sufficiently large) time $t$ and starting
point $z \in [-vt,0]$.
Once more, the result is stated using the
representation in \eqref{eq:w^y}.

\begin{lemma}
  \label{lemma:wave}
  There exists $v_2>\es+2>0$ such that for each $u>0$ and each $v>v_2$ there
  exist $\Delta_0=\Delta_0(u,v) \in \N$ and a $\PP$-a.s.~finite random
  variable $\mathcal{T}=\mathcal{T}(u,v) \geq 0$ so that, $\PP$-a.s., for all
  $t \geq \mathcal{T}$, $\Delta\in \{\Delta_0 ,\Delta_0+1,\dots\}$ and
  $y \in \{0,\dots,\ceil{vt}\}$,
  \begin{equation*}
    w^y(t,\floor{y-vt} ) \geq w^{y+\Delta}(t+u,\floor{y-vt}).
  \end{equation*}
\end{lemma}
We postpone the proofs of Lemmas~\ref{lemma:wave_time} and \ref{lemma:wave}
to Section~\ref{sec:proof_prop_and_lemma}.

\subsection{Proof of Theorem~\ref{thm:tightness} assuming
  Lemmas~\ref{lemma:wave_time} and~\ref{lemma:wave}}
\label{sec:proof_tightness}

We now show Theorem~\ref{thm:tightness} by combining the two lemmas from the
previous section with Theorem~\ref{Thm:zero-crossings}. As explained in
Remark~\ref{rk:v_0_vs_v_c}, we prove Theorem~\ref{thm:tightness} under the
minimal assumption \eqref{eq:ellipticity} on the environment. Consequently,
the arguments presented in this section differ from those in Section 7 in
\cite{CDO25}. We clarify the nature of these differences as we develop the
argument.

Let $\varepsilon \in (0,1)$. As in \cite{CDO25} we introduce the quenched
$\varepsilon$-quantile of the distribution of $M(t)$,
\begin{equation}\label{eq:def_x_t}
  x_t\coloneq \sup\{y \in \Z: w^{y}(t,0) \geq \varepsilon\}
  = \sup \{y \in \Z: \mathtt{P}_0^\xi(M(t) \geq y) \geq \varepsilon \},
  \quad t \geq 0.
\end{equation}
If the environment satisfies \eqref{eq:ellipticity}
and $v_0>v_c$, then, by Theorem 2.1 in \cite{CD20}, the maximum of the BRWRE
obeys a functional central limit theorem with speed $v_0$. Consequently, one
deduces that $x_t /t \rightarrow v_0$ and, in particular,
\begin{equation}\label{eq:x_t}
  \liminf_{t \rightarrow \infty} x_t
  = \infty \quad\text{ and }  \quad
  \limsup_{t \rightarrow \infty} \frac{x_t}{t}< \infty,
  \quad \PP\text{-a.s.}
\end{equation}
It is a priori not clear whether \eqref{eq:x_t} remains valid when the
assumption $v_0>v_c$ is removed.
(Note also that we cannot directly use the law of large numbers for the maximum of the discrete-time BRWRE established in \cite{CP07}, as  the law of $M(t)$ has unbounded support for any $t \in (0,\infty)$.)
As preparation for the proof
of Theorem~\ref{thm:tightness}, the following two claims show that
\eqref{eq:ellipticity} directly implies \eqref{eq:x_t}.

\begin{claim}
  \label{claim:wave}
  Let $x,y \in \Z$ and $w$ be a solution to \eqref{eq:F-KPP}. If $w(0,y) > 0$ for some $y \in \Z$, then
  $\lim_{t \to \infty} w(t,x) = 1$, $\PP$-a.s.
  In particular, $\lim_{t \to \infty} w^y(t,x) = 1$,
  $\PP$-a.s. As consequence,
  \begin{equation}\label{eq:lim_x_t_claim}
    \liminf_{t \rightarrow \infty} x_t =\infty, \quad \PP\text{-a.s.}
  \end{equation}
\end{claim}

\begin{proof}
  By Lemma 6.8 in \cite{CD20}, the number of particles
  in the BRWRE
  at the origin grows exponentially over time. More precisely, for all $x \in \Z$,
  \begin{equation}\label{eq:N(t,0)}
     \mathtt{P}_x^\xi(N(t/2,x) \leq t^2) \leq \mathtt{P}_x^\ei(N(t/2,x)\leq t^2)
    \xrightarrow[]{t \rightarrow \infty}0, \quad \PP\text{-a.s.}
  \end{equation}
  (Here $\mathtt{P}_x^\ei$ denotes $\mathtt{P}_x^\xi$ with $\xi \equiv\ei$.)
  For given $x$ and $y$, if the number of particles at $x$ at time $t/2$ is
  at least $t^2$, then the number of particles reaching $y$ at time $t$ is
  stochastically bounded from below by a binomial random variable
  $B_{t,x,y}\sim$~Bin~$(\floor{ t^2 },p_t)$ with parameters
  $\floor{ t^2}  $ and $p_t\coloneq P_x(X_{t/2}=y) \in
  (0,1)$. By \eqref{eq:N(t,0)}, and since $\floor{ t^2 }
  p_t>t^{5/4}$ for $t\geq0$ large enough,
  \begin{equation*}
    \mathtt{P}^{\xi}_x(N(t,y) \leq t) \leq
    \mathtt{P}^{\xi}_x(N(t/2,x) \leq t^2)
    + \mathtt{P}^{\xi}_x( B_{t,x,y} \leq t)
    \xrightarrow{t \rightarrow \infty} 0,
    \quad \PP\text{-a.s.}
  \end{equation*}
  By combining this observation with \eqref{eq:McKean}, we get
  \begin{equation*}
    w(t,x) \geq 1- \mathtt{E}_x^{\xi}\big[(1-w(0,y))^{N(t,y)}\big]
    \geq 1- \mathtt{P}_x^{\xi}(N(t,y) \leq t) - (1-w(0,y))^t
    \xrightarrow[t \rightarrow \infty]{\PP\text{-a.s.}} 1
  \end{equation*}
  whenever $w(0,y) > 0$ for some $y \in \Z$. This proves the first part of
  the claim.

  To prove \eqref{eq:lim_x_t_claim}, assume by contradiction that
  $\liminf_{t \rightarrow \infty}{x_t}< C-1$ for some $C=C(\xi)\in \N$. Then
  there exist $t_n \rightarrow \infty$ satisfying $x_{t_n} +1\leq C$, so that
  \begin{equation*}
    w^C(t_n,0) = \mathtt{P}^\xi_0(M(t_n) \geq C)
    \leq \mathtt{P}^\xi_0(M(t_n) \geq x_{t_n} +1)
    = w^{x_{t_n}+1}({t_n},0) < \varepsilon.
  \end{equation*}
  Hence, $\limsup_{n \rightarrow \infty} w^C(t_n,0) < 1$, which contradicts
  the first part of the claim.
\end{proof}

\begin{claim}
  \label{claim:v_2}
  Let $\es \in (0,\infty)$ be as in \eqref{eq:ellipticity}. Then
  $x_t \leq \ceil{(\es+2)t} $ for sufficiently large $t \geq 0$,
  $\PP$-a.s.
\end{claim}

\begin{proof}[Proof]
  By a Chernoff bound on $X_t$, for all $a>1$ and $t \geq 0$ one has
  \begin{equation}\label{eq:Chernoff_X_t}
      P_0(X_t \geq at) \leq e^{-(a-1)t}.
  \end{equation}
  Moreover, for $y \in \Z$, $M(t) \geq y$ if
  and only if $N^\geq(t,y) \geq 1$, so that, by \eqref{eq:Feynman_Kac},
  \begin{equation}\label{eq:bound_N^geq}
    \mathtt{P}_{x}^\xi(M(t)\geq y) \leq \mathtt{E}_{x}^\xi[N^{\geq}(t,y)]
    = E_{x}\Big[ e^{\int_0^t \xi(X_s)ds}; X_t \geq y\Big]\leq e^{\es \cdot
      t}P_{x}(X_t \geq y)
  \end{equation}
  for all $x \in \Z$. In particular,
  \begin{equation}
    \label{eq:max_bound}
    \mathtt{P}^{\xi}_0\big(M(t) \geq \ceil{(\es+2)t}\big)
    \leq e^{\es\cdot t} P_0\big(X_t \geq \ceil{(\es+2)t}\big),
  \end{equation}
  which, by \eqref{eq:Chernoff_X_t}, is bounded by $e^{-t/2}$ for $t\geq 0$ large enough. Since, by assumption, $x_t$ satisfies
  $\mathtt{P}^\xi_0(M(t) \geq x_t) \geq \varepsilon >0$,
  the desired result follows.
\end{proof}

Before turning to the proof of Theorem~\ref{thm:tightness}, we show a
technical result addressing the assumptions required in the second part of
Theorem~\ref{Thm:zero-crossings}.

\begin{claim}
  \label{claim:u_0}
  For given $u\in(0,\infty)$ and $y,z \in \Z$, the function $u_0:\Z \rightarrow \R$ given by
  \begin{equation*}
    x \mapsto u_0(x)\coloneq\bbone_{x \geq z}
    - \mathtt{P}^\xi_x(M(u) \geq y)
  \end{equation*}
  satisfies \eqref{eq:cond_u_0} and $u_0\in \ell^1(\Z)$.
\end{claim}

\begin{proof}[Proof]
  Since $\mathtt{P}^\xi_x(M(u) \geq y) \in (0,1)$ whenever
  $u \in (0,\infty)$, \eqref{eq:cond_u_0} is automatically satisfied. It
  therefore suffices to prove that $u_0\in \ell^1(\mathbb Z)$. Without loss
  of generality, it is enough to consider the case $z=0$. Hence, we have to
  show that
  \begin{align*}
    x \mapsto 1-\mathtt{P}_x^\xi(M(u)\geq y)
    =  \mathtt{P}_x^\xi(M(u)< y) \in \ell^1(\N), \quad
    x \mapsto \mathtt{P}_{x}^\xi(M(u)\geq y) \in \ell^1(-\N).
  \end{align*}
  By symmetry, it suffices to consider the second mapping.
  Moreover, by \eqref{eq:bound_N^geq},
  \begin{equation*}
    \mathtt{P}_{x}^\xi(M(u)\geq y) \leq e^{\es \cdot
      u}P_{x}(X_u \geq y)=e^{\es \cdot u}P_{0}(X_u - \ y \geq -x).
  \end{equation*}
  Since $\sum_{x \leq 1} P_0(|X_u-y| \geq -x) = E_0[|X_u-y|]< \infty$, we
  deduce that $x \mapsto \mathtt{E}_{x}^\xi[N^{\geq}(u,y)] \in \ell^1(- \N)$.
  Therefore, $x \mapsto \mathtt{P}_{x}^\xi(M(u) \geq y) \in  \ell^1(- \N)$,
  which concludes the proof of the claim.
\end{proof}

\begin{proof}[Proof of Theorem~\ref{thm:tightness}]
  It is enough to prove tightness of $(M(t)-m(t))_{t \geq \mathcal{T}_0}$ for
  some $\PP$-a.s.~finite $\mathcal{T}_0\geq0$ (which will be
      chosen later). Indeed, by \eqref{eq:bound_N^geq}, for any
  $\mathcal{T}_0\geq0$ there exists $C=C(\mathcal{T}_0) \in \N$ such that,
  $\PP$-a.s.,
  \begin{equation*}
    \inf_{t \leq \mathcal{T}_0} \mathtt{P}_0^\xi(M(t) \geq -C)
    - \mathtt{P}_0^\xi(M(t) \geq C)
    \geq \inf_{t \leq \mathcal{T}_0} P_0(X_{t}\geq -C) -e^{\es\cdot t}P_0(X_{t} \geq C)
    \geq 1-\varepsilon,
  \end{equation*}
  which implies tightness of  the family $(M(t)-m(t))_{t \leq \mathcal{T}_0}$.

  For a given $\varepsilon \in (0,1/2)$, let $x_t$ be the quenched quantile
  defined in \eqref{eq:def_x_t}. The desired tightness follows if we can find
  $\Delta=\Delta(\varepsilon)\in \N$ such that, $\PP$-a.s.,
  \begin{equation}
    \label{eq1:proof_tight}
    w^{x_t-\Delta}(t,0)= \mathtt{P}^\xi_0(M(t) \geq x_t - \Delta)
    \geq 1-\varepsilon, \quad \text{for } t \geq \mathcal{T}_0.
  \end{equation}
  Indeed, $x_t- \Delta \leq m(t) \leq x_t$ by \eqref{eq1:proof_tight} and definition of $m(t)$,
  so that
  \begin{equation*}
    \mathtt{P}^\xi_0(|M(t)-m(t)| > \Delta)
    \leq \mathtt{P}^\xi_0(M(t) > (x_t- \Delta) +  \Delta)
    + \mathtt{P}^\xi_0(M(t) < x_t- \Delta) < 2\varepsilon,
  \end{equation*}
  by definition of $x_t$ and \eqref{eq1:proof_tight}.

  We now prove \eqref{eq1:proof_tight}.
  Since $w^{x_t}(t,0) \geq \varepsilon$, by applying Lemma~\ref{lemma:wave_time} we
  obtain $u=u(\varepsilon) \in (0,\infty)$ such that, $\PP$-a.s.,
  \begin{equation*}
    w^{x_t}(t+t',0) \geq 1 - \varepsilon
  \end{equation*}
   for all  $t' \geq u$ and $t\geq0$.
   Hence, \eqref{eq1:proof_tight} follows
  if for any $u\in (0,\infty)$ we can find some
  $\Delta= \Delta(u) \in \N$ such that, $\PP$-a.s.,
  \begin{equation}\label{eq:w(t+u,0)}
    w^{x_t- \Delta}(t,0)
    \geq w^{x_t}(t+u,0), \quad \text{for } t \geq \mathcal{T}_0.
  \end{equation}
  To prove \eqref{eq:w(t+u,0)}, we employ Theorem~\ref{Thm:zero-crossings}
  and Lemma~\ref{lemma:wave}, and study the zeros~of
  \begin{equation*}
    W(s,x)\coloneq w^{x_t- \Delta}(s,x) -w^{x_t}(s+u,x), \quad s
    \geq 0, \ x \in \Z.
  \end{equation*}
  Since $w^{x_t-\Delta}$ and $w^{x_t}$ solve the F-KPP differential equation
  \eqref{eq:F-KPP}, it is straightforward to verify that $W$ solves a
  differential equation of the form \eqref{eq:main_2}, with
  \begin{equation*}
    \kappa(s,x)\coloneq\xi(x)\big(1-w^{x_t-\Delta}(s,x)-w^{x_t}(s+u,x)\big)
    \in [-\es, \es],
  \end{equation*}
  and the initial condition
  \begin{equation*}
    W(0,x)\coloneq\bbone_{x \geq x_t-\Delta}- \mathtt{P}^\xi_x(M(u) \geq x_t).
  \end{equation*}
  By Claim~\ref{claim:u_0}, we deduce that $W(0,\cdot)$ satisfies
  the assumptions of Theorem~\ref{Thm:zero-crossings}. We now show that,  $\PP$-a.s., for all
  $t \geq \mathcal{T}_0$ there exists
  \begin{equation}\label{eq:x^*}
    x^*=x^*(t) \in -\N \quad \text{ such that } \quad W(t,x^*)\geq 0.
  \end{equation}
  By the second part of
  Theorem~\ref{Thm:zero-crossings}, finding such an $x^*$ implies, $\PP$-a.s.,
  \begin{equation*}
    W(t,0)=w^{x_t- \Delta}(t,0) -w^{x_t}(t+u,0)
    \geq 0, \quad  \text{for } t \geq \mathcal{T}_0,
  \end{equation*}
  and therefore \eqref{eq:w(t+u,0)}.

  It remains to show \eqref{eq:x^*}. To this end we apply
  Lemma~\ref{lemma:wave} to $y \approx x_t$. Let $v_2>\es+2$ as in the
  statement of Lemma~\ref{lemma:wave} and $v\coloneq v_2+1$. By
  \eqref{eq:lim_x_t_claim} and Claim~\ref{claim:v_2}, for all $\Delta \in \N$ there exists a $\PP$-a.s.~finite $\mathcal{T}_1=\mathcal{T}_1(\Delta,v)\geq 0$ such that
  $0 \leq x_t-\Delta < \ceil{vt}$  for all $t \geq \mathcal{T}_1$. Therefore, picking $\Delta=\Delta(u,v) \in \N$ and $\mathcal{T}= \mathcal{T}(u,v) \geq 0$ as in the statement
  of Lemma~\ref{lemma:wave}, gives, $\PP$-a.s.,
  \begin{equation*}
    w^{x_t-\Delta}(t,x^*) \geq w^{x_t}(t+u,x^*),
  \end{equation*}
  for $t \geq \mathcal{T}_0\coloneq\mathcal{T}\vee \mathcal{T}_1$ and $x^*\coloneq\floor{x_t-\Delta-vt}<0$.
  This establishes \eqref{eq:x^*} and completes the proof of the theorem.
\end{proof}

\section{Tilting and exponential change of measure}
\label{sec:tilting}

This section serves as a preparation for the proof of Lemma~\ref{lemma:wave}.
We introduce the previously mentioned family of tilted measures associated to
the BRWRE, and derive some of their properties.

We define $\zeta\coloneq\xi-\es$ and $\triangle\coloneq\es-\ei$. Note that
$\zeta(x)\in [-\triangle,0]$ for all $x \in \Z$. We recall that $X$ stands
for the simple random walk on $\mathbb Z$. We write
$H_y\coloneq\inf \{ t \geq 0: X_t = y\}$ for the hitting time of
$y\in \mathbb Z$. For $\eta \le 0$ and $x,y \in \Z$ with $y \geq x$, we
define a probability measure on the stopped $\sigma$-algebra
$\sigma( (X_{t \wedge H_y})_{t \geq 0})$ via
\begin{equation*}
  P_{x,y}^{\zeta, \eta}(A)
  \coloneq \frac{1}{Z_{x,y}^{\zeta,\eta}}
  E_x\bigg[ \exp \Big( \int_0^{H_y} (\zeta(X_s)+\eta)  ds \Big) ; A \bigg],
\end{equation*}
where the normalizing constant is given by
\begin{equation}
  \label{eqn:Z}
  Z_{x,y}^{\zeta,\eta}:
  = E_x\bigg[ \exp \Big( \int_0^{H_y} (\zeta(X_s)+\eta)  ds \Big) \bigg].
\end{equation}
By the Markov property of $X$ under $P_x$,
\begin{equation}\label{eq:Zxy}
  Z_{x,z}^{\zeta,\eta}=Z_{x,y}^{\zeta,\eta}Z_{y,z}^{\zeta,\eta},
  \quad x \leq y \leq z,
\end{equation}
so that, for given $x \in \Z$, the measures $(P_x^{\zeta,\eta})_{y \geq x }$ are consistent. In particular, by the Kolmogorov extension theorem,
these measure extend to a probability measure $P_x^{\zeta,\eta}$ on $\sigma((X_t)_{t \geq 0})$.

The following result provides an explicit interpretation of the process $X$
under the measure $P_x^{\zeta,\eta}$, as a continuous-time random walk with inhomogeneous
transition probabilities and jump rates. To state the result, we first
extend the definition of $Z_{x,y}^{\zeta,\eta}$ to any $x,y \in \Z$
by setting $Z_{x,y}^{\zeta,\eta} = 1/Z_{y,x}^{\zeta,\eta}$ whenever $y <x$.
A direct computation shows that, with this definition, \eqref{eq:Zxy} extends to all
$x,y,z \in \Z$.

\begin{proposition}
  \label{prop:tilted_RW}
  Given $x \in \Z $, $\eta \leq 0$ and $ \zeta: \Z \mapsto [- \triangle, 0]$,
  let
  \begin{equation}\label{eq:lambda}
    \lambda^{\zeta,\eta}(y)\coloneq1-\zeta(y)-\eta \in \big[1,1 +\triangle+ |\eta|\big],
    \quad y \in \Z.
  \end{equation}
  Under the measure $P_x^{\zeta, \eta}$, $X$ is a continuous-time
  nearest-neighbor random walk (started at $x$)  with spatially inhomogeneous
  transition probabilities
  \begin{equation}\label{eq:p_x_q_x}
    p^{\zeta, \eta}(y,y\pm1) \coloneq
    \frac{Z_{y\pm 1,y}^{\zeta,\eta}}{2\lambda^{\zeta,\eta}(y)},
    \quad y \in \Z,
  \end{equation}
  and jump rates $(\lambda^{\zeta,\eta} (y))_{y \in \Z}$.
\end{proposition}

\begin{proof}
  Let $Y=(Y_n)_{n \geq 0}$ be a discrete-time simple random walk and
  $(e_i)_{i \geq 1}$ a family  of
  i.i.d.~Exp$(1)$-distributed random variables independent of $Y$ such that, $P_x$-a.s.,
  \begin{equation*}
    X_t=Y_{N(t)}, \quad N(t)
    =0 \vee \sup\big\{m \geq 1: e_1 + \dots+e_m \leq t\big\}.
  \end{equation*}
  We begin by studying the law of $e_1$ and $Y_1$ under the measure $P_x^{\zeta, \eta}$.
  By definition of the probability measure $P_x^{\zeta,\eta}$,
  \begin{equation*}
    Z_{x,x+1}^{\zeta,\eta} P_x^{\zeta,\eta}(e_1 \geq t)
    = E_x\left[e^{\int_0^{H_{x+1}} (\zeta(X_s)+\eta ) ds}; e_1 \geq t \right].
  \end{equation*}
  By the Markov property of $X$ under $P_x$, this is equal to
  \begin{align*}
    & E_x\Big[E_{Y_1}\Big[e^{\int_0^{H_{x+1}} (\zeta(X_s)+\eta ) ds} \Big]
      e^{\int_0^{e_1} (\zeta(X_s)+\eta ) ds}; e_1 \geq t \Big]
    \\& =   \bigg(\frac{1}{2}+ \frac{1}{2}
      E_{x-1}\Big[e^{\int_{0}^{H_{x+1}} (\zeta(X_s)+\eta )ds }\Big]\bigg)
    \int_{t}^\infty e^{-s+s(\zeta(x)+\eta)}ds
    \\& =\big(1+Z_{x-1,x+1}^{\zeta, \eta} \big)
    \frac{e^{-(1-\zeta(x)-\eta)t}}{2(1-\zeta(x)-\eta)}
    = \big(1+Z_{x-1,x+1}^{\zeta, \eta} \big)
    \frac{e^{-\lambda^{\zeta,\eta}(x)t}}{2\lambda^{\zeta,\eta}(x)}.
  \end{align*}
  In particular, with $t=0$ this gives
  \begin{equation}\label{eq:formula_lambda}
    2\lambda^{\zeta,\eta}(x)Z_{x,x+1}^{\zeta,\eta}=1+Z_{x-1,x+1}^{\zeta,\eta}.
  \end{equation}
  By combining these observations we get
  $P_x^{\zeta,\eta}(e_1 \geq t)  = e^{-\lambda^{\zeta,\eta}(x)t}$. Hence,
  under $P_x^{\zeta, \eta}$, the holding time $e_1$ is exponentially
  distributed with parameter $\lambda^{\zeta,\eta}(x)$. Similarly, using the
  Markov property of $X $ under $P_x$ once more,
  \begin{align*}
    Z_{x,x+1}^{\zeta,\eta} P_x^{\zeta,\eta}(Y_1 = x\pm1)
    &= E_x\Big[e^{\int_0^{H_{x+1}} (\zeta(X_s)+\eta ) ds}; Y_1 = x\pm1 \Big]\\
    &= \frac{1}{2} E_{x\pm1}\Big[e^{\int_{0}^{H_{x+1}} (\zeta(X_s)+\eta )ds }\Big]
    \int_{0}^\infty e^{-s(1-\zeta(x)-\eta)}ds.
  \end{align*}
  Therefore, the first step of the random walk has the transition probabilities
  \begin{align}
    \label{eq:p_x}
    &P_x^{\zeta,\eta}(Y_1=x+1)=\frac{1}{2Z_{x,x+1}^{\zeta,\eta}(1-\zeta(x)-\eta)}=
    \frac{Z_{x+1,x}^{\zeta,\eta}}{2\lambda^{\zeta,\eta}(x)},\\ \label{eq:q_x}
    &P_x^{\zeta,\eta}(Y_1=x-1)=\frac{Z_{x-1,x+1}^{\zeta,\eta}} {2Z_{x,x+1}^{\zeta,\eta}(1-\zeta(x)-\eta)}
    =\frac{Z_{x-1,x}^{\zeta,\eta}}{2\lambda^{\zeta,\eta}(x)}.
  \end{align}
  Note that the last equality in \eqref{eq:q_x} follows from \eqref{eq:Zxy}.

  To complete the proof, it remains to to show that $X$ satisfies the Markov
  property under the measure $P_x^{\zeta, \eta}$. Hence, we
  need to show that, for all $s,t \geq 0$, $x,y,z \in \Z$ and
  $A\in \sigma((X_r)_{r \leq t})$,
  \begin{align}
    \label{eq:markov}
    P_x^{\zeta, \eta}(A, X_t=y, X_{t+s}=z)
    = P_x^{\zeta, \eta}(A,X_t=y ) P_y^{\zeta, \eta}(X_s=z).
  \end{align}
  To this end, let $n >|x|+|y|+|z|$. In the following, for each process $W$
  and each $0 \leq u \leq r$ we let
  $W^{*}_{u,r}\coloneq \sup \{W_s:u \leq s \leq r\}$ and $W^{*}_r\coloneq W^{*}_{0,r}$.

  By the Markov property of $X$ under $P_x$,
  \begin{align*}
    & E_x\left[ e^{\int_0^{H_{n}} (\zeta(X_r) + \eta)dr};
      A,X_t=y,X^*_t<n, X_{t+s}=z, X_{t,t+s}^{*}-y<n \right]
    \\&= E_x\left[ e^{\int_0^t (\zeta(X_r) + \eta)dr};
      A,X_t=y , X^*_t<n\right]
    E_y\left[ e^{\int_{0}^{H_{n}} (\zeta(X_r) + \eta)dr};X_s=z, X^*_s<n\right]
    \\&=E_x\left[ e^{\int_0^t (\zeta(X_r) + \eta)dr}; A,X_t=y , X^*_t<n \right]
    Z_{y,n}^{\zeta, \eta} P_y^{\zeta, \eta} (X_s=z, X^*_s<n ).
  \end{align*}
  In the last step,
  $\{X_s =z, X_s^* <n\} \subset \sigma((X_{t \wedge H_n})_{t \geq 0})$ and
  the definition of $P_y^{\zeta, \eta}$ were used. Moreover, again by the
  Markov property of $X$ under $P_x$,
  \begin{align*}
    & E_x\left[ e^{\int_0^t (\zeta(X_r) + \eta)dr}; A,X_t=y , X^*_t<n\right] E_y \left[ e^{\int_0^{H_n}(\zeta(X_r)+\eta) dr} \right] \\
    &=E_x\left[ e^{\int_0^{t} (\zeta(X_r) + \eta)dr+\int_t^{H_n} (\zeta(X_r) + \eta)dr}; A,X_t=y , X^*_t<n \right] \\
    &= E_x\left[ e^{\int_0^{H_n} (\zeta(X_r) + \eta)dr} ; A,X_t=y , X^*_t<n \right] = Z_{x,n}^{\zeta, \eta} P_{x}^{\zeta, \eta}(A,X_t=y , X^*_t<n).
  \end{align*}
  By combining the two previous displays, we obtain
  \begin{align*}
    & E_x\Big[ e^{\int_0^{H_{n}} (\zeta(X_r) + \eta)dr};
      A,X_t=y,X^*_t<n, X_{t+s}=z, X_{t,t+s}^{*}-y<n \Big] \\
    &= Z_{x,n}^{\zeta, \eta} P_{x}^{\zeta, \eta} (A,X_t=y , X^*_t<n)
    P_y^{\zeta, \eta} (X_s=z, X^*_s<n ).
  \end{align*}
  Dividing both sides
  by $Z_{x,n}^{\zeta, \eta}$ and taking the limit $n \rightarrow \infty$ gives \eqref{eq:markov}.
  This concludes the proof of the proposition.
\end{proof}

By \eqref{eq:p_x}, \eqref{eq:q_x} and \eqref{eq:Zxy}, the ratio of the
transition probabilities satisfies
\begin{equation}\label{eq:ratio_q_x_p_x}
  \frac{P_x^{\zeta,\eta}(Y_1=x-1)}{P_x^{\zeta,\eta}(Y_1=x+1)}
  = Z_{x-1,x+1}^{\zeta,\eta} \in (0,1].
\end{equation}
Hence, under the tilted measure, the random walk exhibits a positive drift.
Since (for given $\eta$) this drift depends on $\zeta$, we can use
\eqref{eq:ratio_q_x_p_x} to compare the law of $X$ under the measures
$P_x^{0,\eta}$, $P_x^{\zeta,\eta}$ and $P_x^{-\triangle, \eta} $. (In the
  following,  $P_x^{\gamma, \eta}$ denotes $P_x^{\zeta, \eta}$ with
  $\zeta \equiv \gamma \leq 0$.) More precisely, in
Corollary~\ref{cor:coupling} below we compare the law of the hitting time
$H_y$ under this three measures. Note that this is particularly useful, as
under the measures $P_x^{0, \eta}$ and $P_x^{-\triangle, \eta}$ the process $X$ is simply a
continuous-time random walk with a constant drift.

The coupling leading to Corollary~\ref{cor:coupling} is presented in the
following lemma. Before stating the result, we recall that the transition
probabilities and jump rates of $X$ under the measure $P_x^{\zeta, \eta}$ are
given by, respectively, \eqref{eq:p_x_q_x} and \eqref{eq:lambda}.

\begin{lemma}
  \label{lemma:coupling}
  For any $x \in \Z$, $\eta \leq 0$ and
  $ \zeta: \Z \mapsto [- \triangle, 0]$, there exist discrete-time
  nearest-neighbor random walks $Y^{0, \eta}$, $Y^{\zeta, \eta}$ and
  $Y^{-\triangle, \eta}$ with transition probabilities given by, respectively,
  \begin{equation*}
    p^{0,\eta}(y,y\pm1),p^{\zeta,\eta} (y,y\pm1) \text{ and } p^{- \triangle,\eta}(y,y\pm1), \quad y \in \Z,
  \end{equation*}
  such that
  \begin{equation}\label{eq:coupling_Y}
    Y^{0, \eta}_0 = Y^{\zeta, \eta}_0
    = Y^{ - \triangle, \eta}_0=x
    \quad \text{and} \quad
    Y^{0, \eta}_n \leq Y^{\zeta, \eta}_n
    \leq Y^{ - \triangle, \eta}_n,  \quad n \geq 1.
  \end{equation}
  Moreover, there exist counting processes
  $N^{0,\eta},N^{\zeta, \eta},N^{- \triangle, \eta}:[0,\infty) \rightarrow
  \N$ with jump size one, such that
  \begin{equation}\label{eq:N}
    N^{0,\eta} \leq N^{\zeta, \eta}\leq N^{- \triangle, \eta}
  \end{equation}
  and so that the processes
  \begin{equation*}
    X^{0, \eta}:
    =\left(Y^{0, \eta}_{N^{0,\eta}(t)}\right)_{t \geq 0}, \ X^{\zeta, \eta}:
    =\left(Y^{\zeta, \eta}_{N^{\zeta,\eta}(t)}\right)_{t \geq 0}   \text{ and
    }\ X^{-\triangle, \eta}:
    =\left(Y^{-\triangle, \eta}_{N^{-\triangle,\eta}(t)}\right)_{t \geq 0}
  \end{equation*}
  have the
  same law as $X$ under $P_x^{0, \eta}$, $P_x^{\zeta, \eta}$
  and $P_x^{- \triangle, \eta}$, respectively.
\end{lemma}

    Before proving Lemma~\ref{lemma:coupling}, we present a corollary of it, which serves as a discrete-space analog of Lemma 4.3 in \cite{CDO25}.
    Because of the differences between the discrete and the continuous-space setting, the resulting statement is slightly weaker than in \cite{CDO25}, but it will be sufficient for our purposes.
\begin{corollary}
  \label{cor:coupling}
  For any $x \leq y$, $\eta \leq 0$, $t \geq 0$ and
  $\zeta: \Z \mapsto [-\triangle,0]$,
  \begin{equation}
    \label{eq:coupling}
    P_x^{0,\eta}(H_y \leq t) \leq P_x^{\zeta, \eta}(H_y \leq t)
    \leq P_x^{-\triangle, \eta}(H_y \leq t).
  \end{equation}

  \begin{proof}[Proof of Corollary~\ref{cor:coupling}]
    We prove the first inequality, the second then follows by an analogous
    argument. By Lemma~\ref{lemma:coupling}, $X^{0,\eta}$ and
    $X^{\zeta,\eta}$ have the same law as $X$ under, respectively,
    $P_x^{0, \eta}$ and $P_x^{\zeta, \eta}$. In particular, it is enough to
    show that
    \begin{equation*}
      \big\{ \exists s \leq t: X^{0, \eta}_s \geq y \big\}
      \subset \big\{ \exists s \leq t: X^{\zeta, \eta}_s \geq y \big\}
    \end{equation*}
    or, equivalently, that
    \begin{equation}
      \label{eq:hitting_Y}
      \big\{ \exists s \leq t: Y^{0, \eta}_{N^{0, \eta}(s) } \geq y \big\}
      \subset \big\{ \exists s \leq t: Y^{\zeta, \eta}_{N^{\zeta, \eta}(s) } \geq y \big\}.
    \end{equation}
    By \eqref{eq:N} and since $N^{\zeta, \eta}$ is a counting process with
    jump size one, we have
    \begin{equation}\label{eq:N(s)_N(t)}
      N^{0, \eta}(s) \leq N^{\zeta, \eta}(s) \leq N^{\zeta, \eta}(t) \text{ if } s \leq t, \text{ and } N^{\zeta, \eta}([0,t]) = \{0,\dots,N^{\zeta, \eta}(t)\}.
    \end{equation}
    If $Y^{0, \eta}_{N^{0, \eta}(s) } \geq y$ for some $s \leq t$ then, by \eqref{eq:coupling_Y}, $Y^{\zeta, \eta}_{N^{0, \eta}(s) } \geq y$. By \eqref{eq:N(s)_N(t)}, this implies
    \begin{equation*}
        Y^{\zeta, \eta}_{N^{\zeta, \eta}(r) } \geq y, \quad \text{ for some } r \leq t.
    \end{equation*}
    We deduce that \eqref{eq:hitting_Y} is satisfied, which concludes the proof.
  \end{proof}

\end{corollary}

\begin{proof}[Proof of Lemma~\ref{lemma:coupling}]
  Since $\zeta \mapsto Z^{\zeta, \eta}_{y-1,y+1}$ is non-decreasing,
  \eqref{eq:ratio_q_x_p_x} implies that
  \begin{equation*}
    \frac{p^{0,\eta}(y,y-1)}{p^{0,\eta}(y,y+1)}
    \geq \frac{p^{\zeta,\eta}(y,y-1)}{p^{\zeta,\eta}(y,y+1)}
    \geq \frac{p^{\triangle,\eta}(y,y-1)}{p^{\triangle,\eta}(y,y+1)}.
  \end{equation*}
  Hence,
  $p^{0,\eta} (y,y+1) \leq p^{\zeta,\eta} (y,y+1) \leq  p^{-
    \triangle,\eta}(y,y+1)$, and the first part of the lemma  follows from a
  standard coupling argument.

  Let $\{e_i\}_{i \geq 1}$ be i.i.d.~Exp$(1)$-distributed random variables
  and, for $c \in \{0,\triangle\}$,
  \begin{equation}
    \label{eq:N_c}
    N^{-c,\eta}(t)\coloneq \sup \{n \in \N: e_1+\dots+e_m \leq t(1+ c-\eta)\}.
  \end{equation}
  It is straightforward to verify that $X^{0,\eta}$ and
  $X^{-\triangle, \eta}$ satisfy the desired properties. Similarly, by letting
  \begin{equation}\label{eq:N_zeta}
    N^{\zeta,\eta}(t):
    = \sup \left\{n \in \N: \frac{e_1}{1- \zeta(Y_0^{\zeta,\eta})-\eta}+
      \dots + \frac{e_m}{1- \zeta(Y_{m-1}^{\zeta,\eta})-\eta}\leq t\right\}
  \end{equation}
  we get that $\big(Y^{\zeta, \eta}_{N^{\zeta,\eta}(t)}\big)_{t \geq 0}$
  has the same law as $X$ under $P_x^{\zeta, \eta}$. We conclude the proof by
  noticing that, by \eqref{eq:N_c} and \eqref{eq:N_zeta},
  $-\triangle \leq \zeta \leq 0$ implies \eqref{eq:N}.
\end{proof}

We conclude this section with a first application of
Corollary~\ref{cor:coupling}, which provides an upper bound for the critical
speed $v_c$ defined in \eqref{eq:def_v_c}. Before doing so, we recall an
important property of $v_c$, which will be useful also in the next section:
By Lemmas~4.2 and~A.1 in \cite{CD20}, for every $v>v_c$ there exists a unique
$\overline{\eta}(v)<0$ satisfying
\begin{equation}\label{eq:veltilting}
  v=\frac{1}{\E\Big[E_0^{\zeta, \overline{\eta}(v) }[H_1]\Big]}.
\end{equation}
(Recall that $\E$ denotes expectation with respect to the environment.) Roughly
speaking this means that for any $v>v_c$ one can choose the parameter $\eta$ so that the
tilted random walk has, on average, speed $v$.
In what follows we denote by
\begin{equation}\label{eq:def_v_gamma}
  v^{\gamma, \eta}\coloneq \frac{1}{E_0^{\gamma,\eta}[H_1]}
\end{equation}
the speed of $X$ under the measure $P_0^{\gamma, \eta}$,
where $\gamma , \eta \le 0$.

\begin{lemma}\label{lemma:velcrit}
   For each $\gamma, \eta \leq 0$,
  \begin{equation}\label{eq:v_gamma}
      v^{\gamma, \eta}=\sqrt{(\gamma+\eta)(\gamma+\eta-2)}.
  \end{equation}
  Moreover, for $v_c \in (0,\infty)$ being be the critical speed defined in
  \eqref{eq:def_v_c},
  \begin{equation} \label{eq:velcrit}
    v_c
    \leq \sqrt{\triangle(2+\triangle)} < \sqrt{\es(2+\es)} < \es +1.
  \end{equation}
\end{lemma}

\begin{proof}
  We first prove \eqref{eq:v_gamma}.
  By \eqref{eq:def_v_gamma} and \eqref{eq:p_x_q_x},
  \begin{equation}
    \label{eq:compute_v_gamma}
    \begin{split}
      v^{\gamma,\eta}&= \frac{1}{E_0^{\gamma,\eta}[H_1]}
      = \lambda^{\gamma,\eta}(0)\big(1-2p^{ \gamma, \eta}(0,-1)\big)
      \\ &=
      \lambda^{\gamma,\eta}(0)
      \Big(1-\frac{2Z_{-1,0}^{\gamma,\eta}}{2\lambda^{\gamma,\eta}(0)}\Big)
      = \lambda^{\gamma,\eta}(0)-Z_{-1,0}^{\gamma,\eta}
      =1-\gamma - \eta-Z_{-1,0}^{\gamma,\eta}.
    \end{split}
  \end{equation}
  In particular, \eqref{eq:v_gamma} follows once we show that
  \begin{equation}\label{eq:Z_gamma_eta}
    Z_{-1,0}^{\gamma,\eta}
    =1- \gamma - \eta-\sqrt{(\gamma+\eta)(\gamma+\eta-2)}.
  \end{equation}
  To prove \eqref{eq:Z_gamma_eta}, we first notice that
  $Z_{x,x+1}^{\gamma,\eta}=Z_{y,y+1}^{\gamma,\eta}$ for all $x,y \in \Z$. In
  particular, by \eqref{eq:Zxy}, we deduce that
  $Z_{-1,1}^{\gamma,\eta}=Z_{-1,0}^{\gamma,\eta}Z_{0,1}^{\gamma,\eta}
  =(Z_{-1,0}^{\gamma,\eta})^2$. By \eqref{eq:formula_lambda}, this implies
  \begin{equation*}
    2(1- \gamma - \eta)Z_{-1,0}^{\gamma,\eta}=1+(Z_{-1,0}^{\gamma,\eta})^2.
  \end{equation*}
  Solving this quadratic equation gives \eqref{eq:Z_gamma_eta}.

  We now prove \eqref{eq:velcrit}.
  Since $\triangle=\es- \ei$, we just need to prove the first inequality. By
  Lemma~A.1 and Proposition A.3 in \cite{CD20},
  \begin{equation}
    \label{eq1:v_c}
    v_c=\lim_{\eta \downarrow 0} \E\Big[E_0^{\zeta, \eta }[H_1]\Big]^{-1}.
  \end{equation}
  Moreover, Corollary~\ref{cor:coupling} implies that
  $\E[E_0^{\zeta,\eta}[H_1]] \geq E_0^{-\triangle,\eta}[H_1]$ for all
  $ \eta \leq 0$.  By combining this observation with \eqref{eq1:v_c}, \eqref{eq:def_v_gamma} and \eqref{eq:v_gamma}, we get
  \begin{align*}
    v_c \leq \lim_{\eta \downarrow 0}  \frac{1}{E_0^{-\triangle,\eta}[H_1]}
    = \lim_{\eta \downarrow 0} v^{- \triangle,\eta} =\lim_{\eta \downarrow 0}\sqrt{(-\triangle+\eta)(-\triangle+\eta-2)}
    = \sqrt{\triangle(2+\triangle)},
  \end{align*}
  which completes the proof of the lemma.
\end{proof}

The bound in Lemma~\ref{lemma:velcrit} implies, in
particular, that $\overline{\eta}(v)$ exists for every speed $v\geq\es +1$,
regardless of the actual value of $v_c$. This observation will be important
in the proof of Lemma~\ref{lemma:wave}.

\section{Proofs of Lemmas~\ref{lemma:wave_time} and \ref{lemma:wave}}
\label{sec:proof_prop_and_lemma}

In this section, we prove Lemmas~\ref{lemma:wave_time} and \ref{lemma:wave}.
We begin with Lemma~\ref{lemma:wave_time},  which addresses the growth of
$w^y(\cdot,0)=\mathtt{P}_0^\xi(M(\cdot ) \geq y)$ by providing a time
$u=u(\varepsilon)>0$ such that $w^y(t,0) \geq \varepsilon$ implies
$w^y(t+u,0)\geq 1-\varepsilon$.
In \cite{CDO25}, the corresponding statement, Corollary 3.6, was
   proved using rather heavy PDE arguments. In the discrete setting, we
   provide a simple probabilistic proof.

\begin{proof}[Proof of Lemma~\ref{lemma:wave_time}]
  By \eqref{eq:N(t,0)}, there exists $t_1=t_1(\varepsilon)>0$ such that
  \begin{equation}\label{eq1:t'}
    \mathtt{P}_0^\xi(N(t',0) \leq t')
    \leq \mathtt{P}_0^\ei(N(t',0) \leq t')
    \leq \varepsilon/2,
    \quad \text{ for all } t' \geq t_1, \ \PP\text{-a.s.}
  \end{equation}
  Moreover, for all $y \in \Z$,
  \begin{equation}\label{eq2:t'}
    \mathtt{P}_0^\xi(M(t+t') <y)
    \leq \mathtt{P}_0^\xi(M(t+t') <y \mid N(t',0) > t')
    + \mathtt{P}_0^\xi(N(t',0) \leq t'),
  \end{equation}
  where, by the Markov property of the BRWRE,
  \begin{equation}\label{eq3:t'}
    \mathtt{P}_0^\xi(M(t+t') <y \mid N(t',0) > t')
    \leq (\mathtt{P}_0^\xi(M(t) <y))^{t'} = (1-w^y(t,0))^{t'}.
  \end{equation}
  Since $w^y(t,0) \geq \varepsilon$, there exists $t_2=t_2(\varepsilon)>0$ such that
  \begin{equation}\label{eq4:t'}
    (1-w^y(t,0))^{t'} \leq(1-\varepsilon)^{t'}
    \leq \varepsilon/2, \quad \text{ for all } t' \geq t_2.
  \end{equation}
  By combining \eqref{eq1:t'}, \eqref{eq2:t'}, \eqref{eq3:t'} and
  \eqref{eq4:t'}, we conclude that, $\PP$-a.s.,
  \begin{equation*}
      w^y(t+t',0)=1-\mathtt{P}_0^\xi(M(t+t') <y) \geq 1- \varepsilon,
  \end{equation*}
  for all $t' \geq u\coloneq t_1 \vee t_2$.
\end{proof}

The proof of Lemma \ref{lemma:wave} closely follow the arguments in Sections
5 and 6 in \cite{CDO25}. To avoid reproducing their essentially identical
parts, we focus on the main differences and only explain the adaptations
required to carry over the arguments from the continuous-space case to the
discrete one.

As preparation for the proof of Lemma \ref{lemma:wave}, the first (and most important) step is to adapt the definitions of the auxiliary velocities in \cite{CDO25} (see (6.1) and (6.2) therein). Specifically, we let
  \begin{equation}\label{eq:def_v_1}
    v_1\coloneq \es +2, \quad v_2\coloneq\inf\{v > v_1+1: |\overline{\eta}(v)| \geq 2v_1+2\}.
  \end{equation}
  (Recall that $\overline{\eta}(v)$ was defined using \eqref{eq:veltilting}.) Note that, by \eqref{eq:velcrit}, we have $v_c < v_1$. Therefore, any speed
  satisfying $v\geq v_1$ admits some $\overline{\eta}(v)<0$ satisfying
  \eqref{eq:veltilting}. In particular, $v_2$ is well-defined and, since
  $v \mapsto \overline{\eta}(v)$ is a continuous decreasing function such
  that $\lim_{v \rightarrow \infty} \overline{\eta}(v)=-\infty$ (see Lemma 4.2 in \cite{CD20}), also finite.

  The definition of $v_2$ in \eqref{eq:def_v_1} is motivated by the following observation: For any
  $v > v_2$ one can choose $\eta<0$ such that
  \begin{equation}
    \label{eq:vel_1}
    2v_1<|\eta|<|\overline{\eta}(v)|-1,
  \end{equation}
  which, by  \eqref{eq:v_gamma}, implies in particular that $\eta$ satisfies
  \begin{equation}\label{eq:v_0_eta}
    v^{0,\eta}=\sqrt{\eta(\eta-2)}=\sqrt{|\eta|(2+|\eta|})>2v_1.
\end{equation}
(This inequality will be very useful later, see statement of Lemma \ref{lemma_3L<K} below.)

The second step in the preparation for the proof of the Lemma
\ref{lemma:wave} consists in observing that each of the statements in
Section~5 in \cite{CDO25} holds also
in the discrete setting, with obvious modifications.
In particular, the perturbation results of Proposition 5.1 in \cite{CDO25} remain valid in the discrete framework considered here. In fact, the arguments in Section 5 in \cite{CDO25} adapt the results in \cite{DS22}.
As these results are an extension of those obtained in \cite{CD20}, they remain applicable in our setting.

To explain how the proof of Lemma~\ref{lemma:wave} is structured, we briefly
recall the main steps in the proof of its continuous-space analog Lemma 6.1
in \cite{CDO25}. In \cite{CDO25}, the authors first make use of Proposition
5.1 to show that Lemma 6.1 follows from the (more technical) result Lemma
6.2. The latter is then proved via two auxiliary results, Lemmas 6.3 and 6.4.

We proceed in reverse order and first establish discrete-space analogs of
Lemmas 6.3 and 6.4 in \cite{CDO25}. By virtue of the results in Section 4.3
in \cite{CD20}, the proof of Lemma 6.4 in \cite{CDO25} carries over to the
discrete case without substantial modification. It therefore suffices to
address Lemma 6.3 in \cite{CDO25}.

To this end, we define, for $K \geq0$, $t\geq 0$ and $y \in \Z$, the
hitting time
\begin{equation*}
  \mathcal{T}_{y,t}\coloneq\inf\{s\geq 0: X_s \geq \ceil{\beta_{y,t}(s)} \},
  \quad \text{where } \beta_{y,t}(s)\coloneq y-v_1(t-s),
\end{equation*}
and the event $\mathcal{G}_K\coloneq \{\mathcal{T}_{y,t} \in [t-K,t]\}$ (cf.~(6.9) and (6.10) in \cite{CDO25}).

By the nature of the differences between the discrete-space setting considered here and that of \cite{CDO25}, both the statement and the proof of
Lemma \ref{lemma_3L<K} below differ slightly from those of its continuous-space analog Lemma 6.3 in \cite{CDO25}.
\begin{lemma}
  \label{lemma_3L<K}
  Let $\eta <0$ be such that $v^{0,\eta} > 2v_1$.
  Then there exists $K_0=K_0(\eta) \in (0,\infty)$ such that, $\PP$-a.s., for all $v> v_1$, $y \in \Z$, $K \geq K_0$, $t \geq K$ and $L \in (0,K/3]$,
  \begin{equation*}
    P^{\zeta, \eta}_{\floor{ y-vt }}(H_y \leq t, \
      \mathcal{T}_{y,t}\leq t-K)
    \leq 2P^{\zeta, \eta}_{\floor {y-vt}}(H_y <t-L).
  \end{equation*}
\end{lemma}
\begin{proof}
  By a straightforward adaptation of the proof of Lemma 6.3 in \cite{CDO25},
  it is sufficient to show that (cf.~(6.24) therein)
  \begin{equation}\label{eq:H_y}
    P^{\zeta, \eta}_{\ceil{y- v_1(t-u)}}(H_y \leq t-u-L) \geq 1/2, \quad \text{whenever } 0 \leq u \leq t-K.
  \end{equation}
  In \cite{CDO25}, \eqref{eq:H_y} is proved via a lower bound on the expectation of $X_{t-u-L}$ under the measure $P^{0, \eta}_{\ceil{y- v_1(t-u)}}$. Our setting requires a different approach, as \eqref{eq:H_y} does not follow directly from a lower bound of this form.

  By Corollary~\ref{cor:coupling} and translation invariance of  $X$ under $P^{0,\eta}_0$,
  \begin{equation}\label{eq1:t-u-L}
    P^{\zeta, \eta}_{\ceil{y- v_1(t-u)}}(H_y \leq t-u-L)
    \geq P^{0,\eta}_0(H_{\floor{ v_1(t-u)}} \leq t-u-L).
  \end{equation}
  Moreover, since $2L/K \leq 2/3$, $t-u \geq K$ and $K-L \geq 2K/3$,
  \begin{align*}
    \frac{5}{3}v_1(t-u-L) &\geq v_1(t-u-L)+ v_1\frac{2L}{K}(t-u-L)\\
    &\geq v_1(t-u) -v_1 L +v_1 \frac{2L}{K}(K-L) \\
    &\geq v_1(t-u) -v_1 L +v_1 \frac{4L}{3} \geq v_1(t-u).
  \end{align*}
  Hence,
  \begin{equation*}
    v^{0,\eta}(t-u-L) \geq 2v_1 (t-u-L) \geq v_1(t-u) + \frac{v_1}{3}(t-u-L),
  \end{equation*}
  so that
  \begin{equation}\label{eq2:t-u-L}
    t-u-L \geq \frac{v_1(t-u)}{v^{0,\eta}-v_1/3}.
  \end{equation}
  Since, by \eqref{eq:def_v_gamma}, $X$ has speed $v^{0,\eta}$ under the measure $P_0^{0,\eta}$, the law
  of large numbers for the continuous-time random walk gives, by
  \eqref{eq2:t-u-L},
  \begin{equation*}
    P^{0,\eta}_0\left(H_{\floor{ v_1(t-u)}} \leq t-u-L \right)
    \geq P^{0,\eta}_0\left(H_{\floor{ v_1(t-u)}} \leq
      \frac{v_1(t-u)}{v^{0,\eta}-v_1/3} \right) \geq 1/2
  \end{equation*}
  whenever $t-u \geq K \geq K_0$ and $K_0=K_0(\eta)$ is sufficiently large. By
  \eqref{eq1:t-u-L}, this gives the desired relation \eqref{eq:H_y}, and completes the proof.
\end{proof}
We can now state and prove a discrete analog of Lemma 6.2 in \cite{CDO25}.
\begin{lemma}\label{lemma:good_event}
    For every $v>v_2$ there exist constants $K=K(v)\in(0,\infty)$ and $C=C(v) \in (0,\infty)$ such that, $\PP$-a.s.,
    \begin{equation*}
        E_{\floor{y-vt}}\big[e^{\int_0^t \xi(X_s)ds}; X_t \geq y\big] \leq CE_{\floor{y-vt}}\big[e^{\int_0^t \xi(X_s)ds}; X_t \geq y, \mathcal{G}_K\big]
    \end{equation*}
    for all $t \geq 0$ large enough and all $y \in \{0, \dots,\ceil{vt}\}$.
\end{lemma}
\begin{proof}
By a straightforward adaptation of the proof of Lemma 6.2 in \cite{CDO25}, the desired result follows once we establish, for suitably chosen $\eta \leq 0$, $ L \geq 0$, $K \geq0$ and $\delta=\delta(K,\eta)>0$, $C=C(\eta, K)< \infty$, that, $\PP$-a.s.,
\begin{equation}\label{eq:p_y(s)}
    p_y^{\zeta, \eta}(s)\coloneq P_y^{\zeta,\eta}(X_s \geq y) \geq \delta, \quad \text{for all }y \in \Z \text{ and } s \leq K,
\end{equation}
and
\begin{equation}\label{eq:Lemma_6.2}
    P^{\zeta, \eta}_{\floor{y-vt}}(H_y \in [t-L,t]) \leq C P^{\zeta, \eta}_{\floor{y-vt}}(H_y \in [t-K,t], \mathcal{T}_{y,t} \geq t-K), \quad \PP \text{-a.s.},
\end{equation}
for all $t$ large enough and  $y \in \{0,\dots, \ceil{vt} \}$ (cf.~(6.17) and (6.18) in \cite{CDO25}).

However, \eqref{eq:p_y(s)} follows directly from \eqref{eq:lambda}, as
  \begin{equation*}
     P_y^{\zeta,\eta}(X_s \geq y)
    \geq P_y^{\zeta,\eta}(X_r=y, \ \forall \ r \leq K )
    \geq e^{-(1+\triangle + |\eta| )K} \eqcolon \delta>0,
  \end{equation*}
   for all $y \in \Z$  and $s \leq K$. Moreover, by \eqref{eq:vel_1} and $\eqref{eq:v_0_eta}$, we can choose $\eta,L$ and $K$ so that both Lemma \ref{lemma_3L<K} and Lemma 6.4 in \cite{CDO25} (which, we recall, extend to the discrete setting) are applicable. By the exact same arguments as in the proof of Lemma 6.2 in \cite{CDO25}, this gives the desired inequality \eqref{eq:Lemma_6.2}.
\end{proof}

\begin{proof}[Proof of Lemma~\ref{lemma:wave}]
In \cite{CDO25}, Lemma 6.1 is established via two intermediate inequalities, see (6.5) and (6.14) therein. The arguments leading to the first one extend directly to the discrete setting, since (as explained before) Proposition 5.1 in \cite{CDO25} continues to hold in our framework.
On the other hand, the arguments leading to the second intermediate inequality require a small adaptation, as it is a priori not clear whether, in our setting,
\begin{equation}\label{eq1:G_K}
    \int_{0}^{t-K} e^{\es(t-s)}P_{X_s}(X_{t-s} \geq y)ds \leq 1, \quad  \text{for all } K \geq 1, \text{ on the event } \mathcal{G}_K.
\end{equation}
(cf.~(6.11) in \cite{CDO25} and the display following that equation.)
To solve this issue, we first notice that, on the event $\mathcal{G}_K$,
\begin{equation*}
    X_s \leq \floor{y-v_1(t-s)}= y- \ceil{v_1(t-s)}, \quad \text{ for all } s \in [0, t-K).
\end{equation*}
In particular, on the event $\mathcal{G}_K$, every $s \in [0, t-K)$ satisfies
\begin{equation}\label{eq2:G_K}
     P_{X_s}(X_{t-s} \geq y) \leq P_{y- \ceil{v_1(t-s)}}(X_{t-s} \geq y) \leq  P_{0}(X_{t-s} \geq v_1(t-s)).
\end{equation}
Moreover, by \eqref{eq:Chernoff_X_t} and \eqref{eq:def_v_1},
\begin{equation}
  \label{eq3:G_K}
  \begin{aligned}
    \int_0^{t-K}e^{\es(t-s)}P_0(X_{t-s} \geq v_1 (t-s)) ds &
  \leq \int_0^{t-K}e^{\es(t-s)} e^{-(t-s)(v_1-1)}ds
  \\ &\leq \int_K^t e^{-(v_1-1-\es)s}ds \leq e^{-K}.
  \end{aligned}
\end{equation}
By combining \eqref{eq2:G_K} and \eqref{eq3:G_K}, we obtain \eqref{eq1:G_K}.

Using \eqref{eq1:G_K}, we can replicate the arguments in \cite{CDO25} and deduce that (6.14) therein extend to our setting.
With a discrete counterpart to equation (6.14) and Lemma~\ref{lemma:good_event} at our disposal, the last part of the proof of Lemma 6.1 in \cite{CDO25} carries over verbatim to our framework, which concludes the proof.
\end{proof}

\section{Proof of Theorem~\ref{Thm:zero-crossings}}
\label{sec:proof_zero_crossings}

We now prove Theorem~\ref{Thm:zero-crossings}. The proof consists of three parts and, as anticipated in the introduction, follows the approach proposed by Evans and Williams in \cite{EW99} for the real-line setting. First, we prepare for the proof by reducing the statement to a simpler
form. Specifically, we express the solution to
  \eqref{eq:main_2} in terms of a killed random walk by recalling its
  Feynman–Kac representation. Then, we construct a sequence of processes
converging, in a suitable sense, to $t \mapsto u(t,\cdot)$. Finally, by
exploiting the properties of these processes, we conclude the proof of the
theorem.

\subsection{Killed random walk and the Feynman--Kac formula}
\label{sec:KRW}

To prepare for the proof of Theorem~\ref{Thm:zero-crossings}, we begin by
showing that the solution to \eqref{eq:main_2} admits a representation in
terms of a killed random walk. To do so, note first that if
$\norm{\kappa}_\infty < C$ and $u$ solves \eqref{eq:main_2}, then
$\tilde{u}(t,y)=u(t,y)e^{-2Ct}$ is a solution to \eqref{eq:main_2} with
$\kappa+2C \in (C,3C)$ in place of $\kappa$. Because $u$ and $\tilde{u}$
share the same initial condition, zeros and sign properties, we may, without
loss of generality, assume from now on that $\kappa$ is bounded above and
below by positive constants.

We now consider a one-dimensional continuous-time simple random walk with
jump rate one, which is killed with rate $\kappa(t,y)>0$ when located at
position $y \in \Z$ at time $t \geq 0$. This process can be constructed as
follows. Let, as previously, $X$ be a one-dimensional continuous-time simple
random walk with rate one, started at $x \in \Z$ under $P_x$. We denote by
$\dagger$ the point at infinity in the one-point compactification $\OZ$ of
$\Z$. The killed random walk (KRW) is the $\OZ$-valued process
$(\tilde{X}_t)_{t \geq 0}$ defined by
\begin{equation*}
  \tilde{X}_t =
  \begin{cases}
    X_t, & t < \zeta \\
    \dagger, & t \geq \zeta
  \end{cases}
  \quad
  \text{with} \quad
  P_x(\zeta > t)
  = E_x\left[\exp\left(\ - \int_0^t \kappa(s,X_s) ds \right)\right].
\end{equation*}
As the killing rate is bounded away from zero, the life-time $\zeta$ of
$\tilde{X}$ is finite a.s.

To write the unique solution to \eqref{eq:main_2} it terms of the KRW, we use
the Feynman--Kac formula for time-dependent potentials (see Section II.1 in
  \cite{CM94}). Since $u_0$ is summable, the formula applies even if the
initial condition is not non-negative, so that for all $t \geq 0$ and $y \in \Z$
\begin{align*}
  u(t,y) &
  = \ E_y\left[ \exp\left(-\int_0^t \kappa(t-s,X_s)ds\right) u_0(X_t)\right]
  \\ & =  \sum_{x \in \Z} u_0(x)\ E_x\left[ \exp\left(-\int_0^t
      \kappa(s,X_s)ds\right); X_t=y\right]
  =  \sum_{x \in \Z} u_0(x)\ P_{x}(\tilde{X}_t=y).
\end{align*}
At this point, one could deduce
  Theorem~\ref{Thm:zero-crossings} from the previous display by combining the
  Karlin-McGregor determinant formula of coincidence probabilities in
  \cite{Kar88} with well-known variation diminishing properties of strictly
  totally positive matrices (see, for example, Chapter 3 in \cite{Pin10}).
  Nevertheless, we believe that the arguments we present in the following
  might be of independent interest, as they do not involve determinants, but
  rather a purely probabilistic study of particle systems.

We conclude this section with a few reductions of the statement of the
theorem. First, by rescaling $u_0 \in \ell^1(\Z)$ by its $\ell^1$-norm, we
may assume from now on that $|u_0|$ defines a probability measure on $\Z$.
Second, we claim that it is enough to proof the monotonicity of the number of
zero-crossings in the case $s=0$. Indeed, by the above representation of $u$
and the Markov property,
\begin{equation*}
  u(t, \cdot )=\sum_{x \in \Z} u(s,x)P_0(\tilde{X}_{t} = \cdot \  | \tilde{X}_s =x),
\end{equation*}
for each $0 \leq s \leq t$. But $u(s,\cdot) \in \ell^1(\Z)$, and
$(\Tilde{X}_{s+t})_{t \geq 0}$  conditional on $\tilde{X}_s=x$ is again a
KRW, now started at $x$.

\subsection{Annihilating particles}
\label{sec:annihilatingparticles}

The first part of this section provides the main tool for the proof of
Theorem~\ref{Thm:zero-crossings}, namely a sequence of measure-valued
processes that converge, in a suitable sense, to the solution of
\eqref{eq:main_2}.

In the following, we let $\mathcal{M}$, $\overline{\mathcal{M}}$ be,
respectively, the spaces of finite signed measures on $\Z$, $\OZ$, and
$\mathcal{N} \subset \mathcal{M}$,
$\overline{\mathcal{N}} \subset \overline{\mathcal{M}}$ the corresponding
subspaces containing the integer-valued ones. All spaces are equipped with
the weak topology.

For every $\nu \in \mathcal{N}$, there exists a unique choice of integers
$x_1, \dots,x_n \in \Z$ and signs
$\varepsilon_1,\dots, \varepsilon_n \in \{-1,1\}$ such that
\begin{equation*}
  \nu=\sum_{i=1}^n \varepsilon_i \delta_{x_i}, \ x_1\leq \ \dots
  \leq x_n \text{ and }  x_i < x_{i+1} \text{ whenever } \varepsilon_i
  \neq \varepsilon_{i+1}.
\end{equation*}
We denote this unique sequence of signs with
$\mathcal{S}(\nu) \coloneq (\varepsilon_1, \ldots, \varepsilon_n)$. For
$\Tilde{\nu} \in \overline{\mathcal{N}}$ we define
$\mathcal{S}(\Tilde{\nu})=\mathcal{S}(\Tilde{\nu}|_\Z)$.

Our first goal is to construct, for every $\nu \in \mathcal{N}$, a
$\overline{\mathcal{N}} \times \overline{\mathcal{N}}$-valued process $(Y,Z)$
that describes the behavior of $n$ interacting signed particles on the
integer line. More specifically, $Z$ records the positions of the
particles---which we start at $x_1,\dots,x_n$---when considered as unsigned
and non-interacting, while $Y$ incorporates the sign information and the
corresponding interaction structure. Later, the number of particles will be
sent to infinity.

The explicit construction of the process goes as follows. For a given
$\nu \in \mathcal{N}$, let $\tilde{X}^1,\dots,\tilde{X}^n$ be independent
copies of the KRW from the previous section, starting at $x_1,\dots,x_n$. We
define $Z$ via $Z_t \coloneq \sum_{i=1}^n \delta_{\tilde{X}_t^i}$. The process $Y$
is constructed from $Z$ by first assigning to each particle the sign
$\varepsilon_i$, so that
$Y_0=\sum_{i=1}^n \varepsilon_i \delta_{\tilde{X}_0^i}$, and then
annihilating any two particles with opposite sign upon meeting at any site
other than the cemetery $\dagger$. As annihilating two such signed particles
has the same effect as freezing them, the process can be written as
$Y_t=\sum_{i=1}^n \varepsilon_i \delta_{\tilde{X}_{t \wedge \tau_i}^i}$ for
suitable stopping times $\tau_1,\dots,\tau_n \in (0,\infty]$.

We now explain how these stopping times are constructed. The annihilation
time $\tau_i$ of the $i$-th particle is defined as the first time it
encounters another particle, not yet annihilated, with opposite sign. If an
alive particle $\delta_{\tilde{X}^i}$ (not yet annihilated nor sent to the
  cemetery) with sign $\varepsilon_i$ jumps to a site $x \in \Z$ where
possibly multiple (not yet annihilated) particles, say
$\delta_{\tilde{X}^{j_1}},\dots,\delta_{\tilde{X}^{j_m}}$, having sign
$-\varepsilon_i$ are located, then only one of them (say the one with the
  smallest index) will annihilate with $\delta_{\tilde{X}^i}$, which, for the
corresponding stopping times, means $\tau_i=\tau_{j_1} < \tau_{j_k}$ for
$2 \leq k \leq m$. The independence of the random walks implies that their
jumps---and hence the annihilations---almost surely occur at distinct times.
Consequently, if $\tau_i = \tau_j  = \tau_k $ and $i \neq j \neq k$, then
$\tau_i=\tau_j  = \tau_k=\infty$ almost surely, which is possible as no
annihilation occurs at the cemetery.

Our next focus is the study of the zero-crossings of the constructed process.
To this end, we extend the notion of zero-crossing to sign sequences and
finite signed measures on $\Z$ by identifying them with sequences in $\ell^1(\Z)$.
The definition can be extended to any finite signed measure $\mu$ on $\OZ$
via $\Sigma(\mu) \coloneq \Sigma(\mu|_{\mathbb{Z}})$.

By considering the process $(\mathcal{S}(Y_t))_{t \geq 0}$, which  records
the signs of the particles together with their ordering, we notice that
$\Sigma(\mu) = \Sigma(\mathcal{S}(\mu))$ for any
$\mu \in \overline{\mathcal{N}}$. In particular, it is sufficient to describe
the time evolution of the zero-crossings of $\mathcal{S}(Y_t)$. To do so, we
introduce the notion of a substring. For
$\mu, \tilde{\mu} \in \overline{\mathcal{N}}$, we say that $\mathcal{S}(\mu)$
is a substring of the sign sequence $\mathcal{S}(\tilde{\mu})$, denoted
$\mathcal{S}(\mu) \preceq \mathcal{S}(\tilde{\mu})$, if $\mathcal{S}(\mu)$ is
either identical to $\mathcal{S}(\tilde{\mu})$ or can be obtained from it by
removing finitely many signs.

The following result, which will be used in the next section, serves as our
counterpart to part $(iv)$ of Lemma 3.1 in \cite{EW99}. The
  proof can be adapted to our framework because particles move via
nearest-neighbor jumps, a property that plays the role of path continuity in
the original argument.

\begin{lemma}
  \label{lemma:zero-crossings}
  Almost surely, for each $s \leq t$
  \begin{equation*}
    \mathcal{S}(Y_t) \text{ is a substring of } \mathcal{S}(Y_s)
    \quad \text{ and } \quad
    \Sigma(Y_t) \leq \Sigma(Y_s) \leq \Sigma(\nu).
  \end{equation*}
\end{lemma}

\begin{proof}
  Almost surely, each particle has a finite life-time and therefore jumps
  finitely many times. Therefore, there are finitely many collisions, and the
  right-continuous $\overline{\mathcal{N}}$-valued process $(Y_t)_{t \geq 0}$
  is almost surely piecewise constant with finitely many jumps. The same
  holds for the process $(\mathcal{S}(Y_t))_{t \geq 0}$. Moreover, if $s$ and
  $t$ are such that the process has exactly one jump in the time interval
  $(s,t]$, then $\mathcal{S}(Y_t)$ is  obtained from $\mathcal{S}(Y_s)$ by
  the removal of either one sign, in the case where a particle is sent to the
  cemetery, or two consecutive signs, in the case of a particle pair
  undergoing annihilation. In fact, since particles do not jump
  simultaneously, $\mathcal{S}(Y_t)$ cannot be obtained by permuting two
  signs in $\mathcal{S}(Y_s)$. In particular, one has
  $\mathcal{S}(Y_t) \preceq \mathcal{S}(Y_s)$. By transitivity, we conclude
  that $\mathcal{S}(Y_t) \preceq \mathcal{S}(Y_s)$ for all $0 \leq s \leq t$.
  The second part of the lemma is a direct consequence of the first one.
\end{proof}

We now construct a sequence of $\overline{\mathcal{N}}$-valued processes
$(Y^n)_{n \in \N}$, each representing the dynamics of $n$ particles evolving
under the interaction and annihilation mechanism of the process $Y$ above. We
add an additional layer of randomness by choosing random initial
configurations $(Y_0^n)_{n \in \N}$.

Let $(\mathcal{X}_i)_{i \in \N}$ be an i.i.d.~sequence of $|u_0|$-distributed
starting positions, with corresponding signs
$\mathcal{E}_i\coloneq\sgn(u_0(\mathcal{X}_i))$. We recall that we assume  $|u_0|$
is a probability measure. Then, defining
\begin{equation*}
  Y_0^n\coloneq \sum_{i=0}^n \mathcal{E}_i \delta_{\mathcal{X}_i}
\end{equation*}
yields, by repeating the previous construction with $Y_0^n$ in place of
$\nu$, a sequence of $\overline{\mathcal{N}}$-valued processes
$(Y^n)_{n \in \N}$.

Once normalized, these processes converge (in a suitable sense) to the
deterministic function $t \in [0,\infty) \rightarrow u(t, \cdot)$, which can
be extended to $\OZ$ by setting
$u(t,\dagger)\coloneq\sum_{x \in \Z} u_0(x) P_x(\tilde{X}_t = \dagger)$ and thus
can be viewed as an element of $ \overline{\mathcal{M}}$.

\begin{lemma}
  \label{lemma:TildeYn}
  The sequence $(\frac{1}{n}Y^n)_{n \in \N}$ of càdlàg
  $\overline{\mathcal{M}}$-valued processes converges in probability in the
  Skorokhod topology to the continuous deterministic function
  $t \mapsto u(t,\cdot) \in \overline{\mathcal{M}}$.
\end{lemma}

The proof of Lemma~\ref{lemma:TildeYn} is a straightforward adaptation of the
arguments in the proof of Lemma 4.1 in \cite{EW99}. The crucial observation,
which makes the adaptation to our setting straightforward, is that our KRW
induces a Feller semigroup $(\mathcal{P}_t)_{t \geq 0}$ of linear operators
on the space or continuous functions on $[0,\infty)\times \Z$. More
precisely, the family $(P_{(s,x),t})_{s,t\geq0, x \in \Z}$ of subprobability
measures on $[0,\infty)\times \Z$, given by
\begin{equation*}
  P_{(s,x),t}(A \times B)
  \coloneq\delta_{s+t}(A) P_x(\tilde{X}_{s+t} \in B | \tilde{X}_s=x),
\end{equation*}
induces a semigroup with the desired properties via
\begin{equation*}
  \mathcal{P}_t(f)(s,x)\coloneq\int P_{(s,x),t}(dz) f(z)
  = \sum_{y \in \Z} f(s+t,y) P_x(\tilde{X}_{s+t} = y | \tilde{X}_s =x).
\end{equation*}

\begin{proof}[Proof of Theorem~\ref{Thm:zero-crossings}]
  Building on the above results, we conclude the proof of
  Theorem~\ref{Thm:zero-crossings} via a minor modification of the arguments
  in Section 4 in \cite{EW99}.

  We begin by establishing \eqref{eq:u_t_vs_u_s}. By
  Lemma~\ref{lemma:TildeYn}, we can pick a subsequence
  $(\frac{1}{n_k}Y^{n_k})_{k \in \N}$ converging almost surely to
  $(u(t,\cdot))_{t \geq 0}$ and hence, by continuity of the
    function $t \mapsto u(t,\cdot) \in \overline{\mathcal{M}}$ with respect
    to the weak topology, such that for all $t \geq 0$ the random measures
  $(\frac{1}{n_k}Y_t^{n_k})_{k \in \N}$ converge almost surely to
  $u(t,\cdot)$. In particular, if $\Sigma(u(t,\cdot))= m$ then, almost
  surely, there exist $x_1<\dots<x_{m+1} \in \Z$ and a (random)
  $\mathcal{K}_0 \in \N$ such that
  \begin{equation*}
    \frac{Y_t^{n_k}(x_i)}{n_k}\frac{Y_t^{n_k}(x_{i+1})}{n_k}
    <0, \quad \text{ for } k \geq \mathcal{K}_0 \text{ and } 1 \leq i \leq m,
  \end{equation*}
  which implies that $\liminf_{k \rightarrow \infty} \Sigma(Y_t^{n_k}) \geq m$.
  Therefore, almost surely,
    \begin{equation*}
      \Sigma(u(t,\cdot)) \leq \liminf_{k \rightarrow \infty} \Sigma(Y_t^{n_k}).
    \end{equation*}
    By Lemma~\ref{lemma:zero-crossings}, the right-hand side is almost surely
    bounded by $\liminf_{k \rightarrow \infty}\Sigma(Y^{n_k}_0)$, which, by
    construction, coincides with $\Sigma (u_0)$.

  To prove the second part of the theorem, we first show the following result.

  \begin{claim}
    \label{claim}
    Under the assumptions of Theorem~\ref{Thm:zero-crossings}, if $u_0$ is
    such that $u_0(x)<0$ implies $u_0(y)\leq 0$ for each $y <x$  and
    $u_0(x)>0$ implies $u_0(y)\geq0$  for each  $y >x$, then $u(t,\cdot)$
    satisfies the same property for every $t \geq 0$.
  \end{claim}

  Even though this property seems to be a direct consequence of
  \eqref{eq:u_t_vs_u_s}, we need to rule of the possibility that the two
  signs in $\mathcal{S}(\sgn(u(t,\cdot)))$ change order in time.

  \begin{proof}[Proof of Claim~\ref{claim}]
    We argue by contradiction, and assume that $u(t,x)>0>u(t,y)$ for some
    $y>x$. Then, by repeating the above arguments, $Y_t^{n}(x)>0>Y_t^{n}(y)$
    for sufficiently large $n \in \N$, almost surely. In particular,
    $(+1,-1) \preceq \mathcal{S}(Y_t^n)$. Since
    $\mathcal{S}(Y_t^n) \preceq \mathcal{S}(Y_0^n)$ by
    Lemma~\ref{lemma:zero-crossings}, one can find $y_0>x_0$ with
    $Y_0^{n}(x_0)>0>Y_0^{n}(y_0)$. By construction of $Y_0^n$, this
    contradicts the assumptions on $u_0$.
  \end{proof}

  We now use this claim to derive a stronger result, namely that $u(t,y) >0$
  implies $u(t,x) >0$ for all $x>y$. Assume by contradiction that there
  exists $x>y$ with $u(t,x) = 0 < u(t,y)$ (by Claim~\ref{claim}, $u(t,x)$
    must be non-negative). If we set
  \begin{equation*}
    0<\varepsilon \coloneq
    \begin{cases}
      \frac{1}{2}u(t,y)\wedge u_0(x), & \text{ if } u_0(x)> 0 \\
      \frac{1}{2}u(t,y), & \text{ if } u_0(x) \leq 0
    \end{cases}
  \end{equation*}
  and consider the solution
  \begin{equation*}
    u^\varepsilon(t, \cdot)
    =-\varepsilon P_x(\tilde{X}_t= \cdot\ ) \in (-\varepsilon,0), \quad t >0,
  \end{equation*}
  of \eqref{eq:main_2} with initial condition
  $u_0^\varepsilon\coloneq-\varepsilon\bbone_x$, then
  $u(t,\cdot)+u^\varepsilon(t,\cdot)$ describes a solution of
  \eqref{eq:main_2} with initial condition
  $u_0+u_0^\varepsilon \in \ell^1(\Z)$. By the construction of $\varepsilon$
  and the assumptions on $u_0$, $u_0+u_0^\varepsilon \in \ell^1(\Z)$
  satisfies the assumptions of Claim~\ref{claim}. However, since
  $u(t,y)+u^\varepsilon(t,y) \geq \frac{1}{2}u(t,y)>0$ and
  $u(t,x)+u^\varepsilon(t,x) <0$, the claim yields a contradiction.
  Therefore, $\{x: u(t,x)>0\}$ is necessarily a (possibly empty) set of the
  form $\{b_t,b_t+1,\dots,\}$. A similar argument can be carried out for the
  negative case. This concludes the proof of
    Theorem~\ref{Thm:zero-crossings}.
\end{proof}

We conclude with a few remarks on possible generalizations of
Theorem~\ref{Thm:zero-crossings}. Replacing $\tilde{X}$ with a killed simple
random walk with non-zero drift and arbitrary jump rate does not affect the
arguments presented above. In particular, Theorem~\ref{Thm:zero-crossings}
can be extended to a broader class of discrete differential equations. More
precisely, if $\alpha>0$ and $\beta \in (-\alpha,\alpha)$ then one can
replace equation \eqref{eq:main_2} in the theorem with
\begin{equation*}
  \partial_tu(t,x) = \alpha \Delta_d u(t,x)+\beta \partial_xu(t,x)-\kappa(t,x)u(t,x),
\end{equation*}
where $\partial_xu(t,x)=u(t,x+1)-u(t,x-1)$. The nearest-neighbor killed
random walk corresponding to this differential equation is the one which
jumps with rate $2\alpha$, has drift $-\beta/\alpha$, and is killed with
rate $\kappa(t,x)$.

On the other hand, the restriction to nearest-neighbor jumps is crucial for
the theorem to hold without additional assumptions on $u_0 \in \ell^1(\Z)$.
Dropping this condition allows simple counterexamples in which the number of
zero-crossings increases over time. For instance, if $u_0=\bbone_{\{0\}}-\bbone_{\{1\}}$
and the random walk satisfies
$\{x\} \subsetneq \text{supp}(P_x(\tilde{X}_t = \cdot \ )) \subset x+2\Z$ for
$x \in \{0,1\}$, then $\Sigma(u(t,\cdot)) > 1 = \Sigma(u_0)$.

\providecommand{\bysame}{\leavevmode\hbox to3em{\hrulefill}\thinspace}
%define \MR and \MRhref to do something useful
\providecommand\MR{}
\renewcommand\MR[1]{\relax\ifhmode\unskip\spacefactor3000
\space\fi \MRhref{#1}{#1}}
\providecommand\MRhref{}
\renewcommand{\MRhref}[2]%
{\href{http://www.ams.org/mathscinet-getitem?mr=#1}{MR#2}}
\providecommand{\href}[2]{#2}

\providecommand{\arxiv}[1]{Preprint, available at \href{http://arxiv.org/abs/#1}{arXiv:#1}}

\end{document}